\documentclass[preprint,12pt]{amsart}
\usepackage{amsmath,amsthm,amsfonts}
\input{amssym.def}
\input{amssym.tex}

\usepackage{graphicx, color}
\numberwithin{equation}{section} 
\theoremstyle{plain}

\usepackage{color}

\usepackage{graphics}
\newcommand{\C}{\mathbb C}
\newcommand{\R}{\mathbb R}
\newcommand{\Z}{\mathbb Z}

\newtheorem{thm}{Theorem}[section]
\newtheorem{lem}[thm]{Lemma}
\newtheorem{pro}[thm]{Proposition}

\newtheorem{rk}[thm]{Remark}

\begin{document}

\title[Volume preserving diffeomorphisms ]
{\bf Volume preserving diffeomorphisms with weak and limit weak
shadowing }



\author{Manseob Lee
 }
\address
     { Manseob Lee :  Department of Mathematics \\
       Mokwon University, Daejeon, 302-729, Korea
      }
\email{lmsds@mokwon.ac.kr  }

\thanks{
\newline
\indent This research was supported by Basic Science Research
Program through the National Research Foundation of Korea(NRF)
funded by the Ministry of Education, Science and Technology (No.
2011-0007649).}

\begin{abstract}
Let $f$ be a volume-preserving diffeomorphism of a closed
$C^{\infty}$ two-dimensional Riemannian manifold $M.$ In this
paper, we prove the equivalence between the following conditions:
\begin{itemize}
\item[(a)] $f$ belongs to the $C^1$-interior of the set of
volume-preserving diffeomorphisms which satisfy the weak shadowing
property.

\item[(b)] $f$ belongs to the $C^1$-interior of the set of
volume-preserving diffeomorphisms which satisfy the limit weak
shadowing property,
\item[(c)] $f$ is Anosov.
\end{itemize}
\end{abstract}

\maketitle


\section{Introduction}

Let $M$ be a closed $C^{\infty}$ $n$-dimensional Riemannian
manifold, and let ${\rm Diff}(M)$ be the space of diffeomorphisms
of $M$ endowed with the $C^1$-topology. Denote by $d$ the distance
on $M$ induced from a Riemannian metric $\|\cdot\|$ on the tangent
bundle $TM$. Let $f:M\to M$ be a diffeomorphism, and let
$\Lambda\subset M$ be a closed $f$-invariant set.

 For $\delta>0$, a sequence of points
$\{x_i\}_{i=a}^{b}(-\infty\leq a< b \leq \infty)$ in $M$ is called
a {\it $\delta$-pseudo orbit} of $f$ if $d(f(x_i),
x_{i+1})<\delta$ for all $a\leq i\leq b-1.$ For given $x, y\in M$,
we write $x\rightsquigarrow y$ if for any $\delta>0$, there is a
$\delta$-pseudo orbit $\{x_i\}_{i=a_{\delta}}^{b_{\delta}}(
a_{\delta}< b_{\delta})$ of $f$ such that $x_{a_{\delta}}=x$ and
$x_{b_{\delta}}=y.$ The set of points $\{x\in M:x\rightsquigarrow
x\}$ is called the {\it chain recurrent set} of $f$ and is denoted
by $\mathcal{CR}(f).$ If we denote the set of periodic points of
$f$ by $P(f)$, then $P(f)\subset\Omega(f)\subset\mathcal{CR}(f).$
Here $\Omega(f)$ is the non-wandering set of $f.$

 We say that $f$ has the {\it shadowing property} on $\Lambda$ if for any
 $\epsilon>0$ there is $\delta>0$ such that for any
 $\delta$-pseudo orbit $\{x_i\}_{i\in\Z}\subset\Lambda$ of $f$ there is $y\in M$
 such that $d(f^i(y), x_i)<\epsilon,$ for $i\in\Z.$
Note that in this definition, the shadowing point $y\in M$ is not
necessarily contained in $\Lambda.$ We say that $f$ has the {\it
$C^1$-interior shadowing property} if there is a
$C^1$-neighborhood $\mathcal{U}(f)$ of $f$ such that for any
$g\in\mathcal{U}(f)$, $g$ has the shadowing property.

The weak shadowing property was introduced in \cite{P}. The weak
shadowing property is investigated in \cite{Pl, P, PST, S1, S2}.
Every diffeomorphism having the shadowing property has the weak
shadowing property but the converse is not true. Indeed, an
irrational rotation on the unit circle has the weak shadowing
property but does not have the shadowing property.

Given $\epsilon>0,$ $\{x_i\}_{i\in\Z}$ is said to be {\it weakly
$\epsilon$-shadowed} by $y\in M$ if $\{x_i\}_{i\in\Z}\subset
B_{\epsilon}(\mathcal{O}_f(y)).$ Here $B_{\epsilon}(A)=\{y\in M:$
there is $x\in A$ such that $d(x, y)<\epsilon\}$ is the
$\epsilon$-neighborhood of a subset $A$ of $M.$ We say that $f$
has the {\it weak shadowing property} if for every $\epsilon>0$,
there is $\delta>0$ such that every $\delta$-pseudo orbit of $f$
can be weakly $\epsilon$-shadowed by some point. Note that $f$ has
the weak shadowing property if and only if $f^n$ has the weak
shadowing property for every $n\in\Z.$ We say that $f$ has the
{\it $C^1$-interior weak shadowing property} if there is a
$C^1$-neighborhood $\mathcal{U}(f)$ of $f$ such that for any
$g\in\mathcal{U}(f)$, $g$ has the weak shadowing property.

Now, we introduce the notion of the limit weak shadowing property
which is introduced in \cite{S3}. We say that $f$ has the {\it
limit weak shadowing property} if for every $\epsilon>0$ there is
$\delta>0$ such that for any $\delta$-limit pseudo orbit
$\{x_i\}_{i\in\Z},$ there exists $y\in M$ weakly
$\epsilon$-shadowing $\{x_i\}_{i\in\Z}$, and, if in addition,
$d(f(x_i), x_{i+1})\to 0$ as $i\to\pm\infty$ then
$d(\mathcal{O}_f(y), x_i)\to 0$ as $i\to\pm\infty.$ Clearly, the
limit weak shadowing property is stronger than the weak shadowing
property by definition. Note that $f$ has the limit weak shadowing
property if and only if $f^n$ has the limit weak shadowing
property for every $n\in\Z.$ We say that $f$ has the {\it
$C^1$-interior limit weak shadowing property} if there is a
$C^1$-neighborhood $\mathcal{U}(f)$ of $f$ such that for any
$g\in\mathcal{U}(f)$, $g$ has the limit weak shadowing property.

 Note that if $f$ is
topologically transitive then $f$ has the weak shadowing property
and $f$ has the limit weak shadowing property.

The shadowing property usually plays an important role in the
investigation of stability theory and ergodic theory.  Sakai
\cite{S1} showed that a diffeomorphism belonging to the
$C^1$-interior of the set of all diffeomorphisms on a $C^{\infty}$
closed surface with weak shadowing property satisfies both Axiom A
and the no-cycle condition. Thus we can restate the above facts as
follows.

\begin{thm}\label{thm00}Let $M$ be a closed two-dimensional manifold.
 A diffeomorphism $f$ belongs to the $C^1$-interior
weak shadowing property if and only if $f$ satisfies both Axiom A
and the no-cycle condition.
\end{thm}

Hence the $C^1$-interior weak shadowing property in the two
dimensional manifold is characterized as the $\Omega$-stability of
the system by Theorem \ref{thm1} $C^1$ non-empty open set
$\mathcal{U}\subset{\rm Diff}(\mathbb{T}^3)$ such that every
$g\in\mathcal{U}$ is topologically transitive but not Anosov. It
is easy to see that every $g\in\mathcal{U}$ has the weak shadowing
property but does not satisfy Axiom A and the no-cycle condition.
 In
\cite{ST}, the authors proved that an $\Omega$-stable
diffeoemorphism has the limit weak shadowing property. And Sakai
\cite{S3} showed that there is a diffeomorphism $f$ on 2 torus
belonging to the $C^1$-interior of the set of diffeomorphisms
possessing the limit weak shadowing property such that $f$ does
not satisfy the strong transversality condition. Thus we can
restate the above facts as follows.
\begin{thm}\label{thm01} Let $\mathbb{T}^2$ be the two dimensional torus.
 There is a diffeomorphism $f$ belongs to the $C^1$-interior
limit weak shadowing property  satisfying both Axiom A and the
no-cycle condition, but not the strong transversality condition.
\end{thm}

By the theorem, even though a diffeomorphism is contained in the
$C^1$-interior of the set of diffeomorphisms possessing the limit
weak shadowing property, it does not necessarily satisfy the
strong transversality condition.

A periodic point $p$ of $f$ is {\it hyperbolic} if $Df^{\pi(p)}$
has eigenvalues with absolute values different of one, where
$\pi(p)$ is the period of $p.$ Denote by $\mathcal{F}(M)$ the set
of $f\in{\rm Diff}(M)$ such that there is a $C^1$-neighborhood
$\mathcal{U}(f)$ of $f$ such that for any $g\in\mathcal{U}(f),$
every $p\in P(g)$ is hyperbolic. It is proved that by Hayashi
\cite{H} that $f\in\mathcal{F}(M)$ if and only if $f$ satisfies
both Axiom A and the no-cycle condition.

Let $\Lambda$ be a closed $f\in{\rm Diff}(M)$-invariant set. We
say that $\Lambda$ is {\it hyperbolic} if the tangent bundle
$T_{\Lambda}M$ has a $Df$-invariant splitting $E^s\oplus E^u$ and
there exists constants  $C>0$ and $0<\lambda<1$ such that
$$\|D_xf^n|_{E_x^s}\|\leq C\lambda^n\;\;{\rm and}\;\;\|D_xf^{-n}|_{E_x^u}\|\leq C\lambda^{n} $$
for all $x\in \Lambda$ and $n\geq 0.$ If $\Lambda=M$ then we say
that $f$ is an {\it Anosov diffeomorphism}.

\section{Statement of the results}

A fundamental problem in differentiable dynamical systems is to
understand how a robust dynamic property on the underlying
manifold would influence the behavior of the tangent map on the
tangent bundle. For instance, in \cite{M}, Ma\~n\'e proved that
any $C^1$ structurally stable diffeomorphism is an Axiom A
diffeomorphism. And in \cite{Pa}, Palis extended this result to
$\Omega$-stable diffeomorphisms.

 Let $M$ be a compact $C^{\infty}$
$n$-dimensional Riemannian manifold endowed with a volume form
$\omega.$ Let $\mu$ denote the measure associated to $\omega,$
that we call Lebesgue measure, and let $d$ denote the metric
induced by the Riemannian structure. Denote by ${\rm
Diff}_{\mu}(M)$ the set of diffeomorphisms which preserves the
Lebesgue measure $\mu$ endowed with the $C^1$-topology. In the
volume preserving, the Axiom A condition is equivalent to the
diffeomorphism be Anosov, since $\Omega(f)=M$ by Poincar\'e
Recurrence Theorem. The purpose of this paper is to do this using
the robust property.

We define the set $\mathcal{F}_{\mu}(M)$ as the set of
doffeomorphisms $f\in{\rm Diff}_{\mu}(M)$ which have a
$C^1$-neighborhood $\mathcal{U}(f)\subset{\rm Diff}_{\mu}(M)$ such
that if for any $g\in\mathcal{U}(f)$, every periodic point of $g$
is hyperbolic. Note that
$\mathcal{F}_{\mu}(M)\subset\mathcal{F}(M)$(see \cite[Corollary
1.2]{AC}).

Very recently, Arbieto and Catalan \cite{AC} proved that if a
volume preserving diffeomorphism contained in
$\mathcal{F}_{\mu}(M)$ then it is Anosov.
 From the
above facts, we can restate as follows.

\begin{thm}\label{thm1} Any diffeomorphism in
$\mathcal{F}_{\mu}(M)$ is Anosov.
\end{thm}

Very recently, Lee \cite{L} showed that if a volume preserving
diffeomorphisms on any dimensional manifold belongs to the
$C^1$-interior expansive or $C^1$-interior shadowing property,
then it is Anosov. As in the above theorems \ref{thm00} and
\ref{thm01}, we can't extend on any dimensional manifold. Thus, we
study the cases when a volume preserving diffeomorphism is in
$C^1$-interior weak shadowing property or $C^1$-interior limit
weak shadowing property on two-dimensional mainfold, then it is
Anosov. Let $int\mathcal{WS}_{\mu}(M)$ be denote the set of volume
preserving diffeomorphisms in ${\rm Diff}_{\mu}(M)$ satisfying the
weak shadowing property, and let $int\mathcal{LWS}_{\mu}(M)$ be
denote the set of volume-preserving diffeomorphisms in ${\rm
Diff}_{\mu}(M)$ satisfying the limit weak shadowing property. Main
thing to prove this paper is the following.

\begin{thm}\label{thm2} Let $M$ be a two-dimensional manifold, and let $f\in{\rm Diff}_{\mu}(M).$ We has that
$$int \mathcal{WS}_{\mu}(M)=
int\mathcal{LWS}_{\mu}(M)=\mathcal{AN}_{\mu}(M),$$ where
$\mathcal{AN}_{\mu}(M)$ is the set of Anosov volume preserving
diffeomorphisms in ${\rm Diff}_{\mu}(M)$.
\end{thm}

\section{Proof of Theorem \ref{thm2}}
Let $M$ be a closed $C^{\infty}$ $n$-dimensional Riemannian
manifold, and let ${\rm Diff}(M)$ be the space of diffeomorphisms
of $M$ endowed with the $C^1$-topology. Denote by $d$ the distance
on $M$ induced from a Riemannian metric $\|\cdot\|$ on the tangent
bundle $TM$. Let $f:M\to M$ be a diffeomorphism, and let
$\Lambda\subset M$ be a closed $f$-invariant set. From now, we
study relation between  a normally hyperbolic(see \cite{HPS}) and
the weak shadowing property as follows lemmas.

\begin{lem}\label{weaknormal}Let $f\in{\rm Diff}(M),$ and
let $\Delta\subset\Lambda$ be a normally hyperbolic $f$-invariant
submanifold of $M.$ Suppose that $f$ has the weak shadowing
property on $\Lambda$. Then the shadowing point is in $\Delta.$
\end{lem}

\begin{proof} Suppose that $f$ has the weak shadowing
property on $\Lambda$. For any $\epsilon>0$, let
$B_{\epsilon}(\Delta)$ be the $\epsilon$-neighborhood of $\Delta.$
Since $\Delta$ is a normally hyperbolic,  we can choose $k>0$ and
$\epsilon_1>0$ such that for any $x\in
B_{\epsilon_1}(\Delta)\setminus\Delta,$ $d(f^k(x), \Delta)>
\epsilon_1.$ Let $0<\delta<\epsilon_1$ be the number of the weak
shadowing property of $f|_{\Lambda}$ for $\epsilon_1.$ Since $f$
has the weak shadowing property on $\Lambda,$ $f$ must have the
weak shadowing property on $\Delta.$ Thus for any $\delta$-pseudo
orbit $\{x_i\}_{\in\Z}\subset\Delta$ of $f,$ we can find a point
$y\in B_{\epsilon_1}(\Delta)$ such that $\{x_i\}_{i\in\Z}\subset
B_{\epsilon_1}(\mathcal{O}_f(y)).$ But, if $y\in
B_{\epsilon_1}(\Delta)\setminus \Delta$ then from the above facts,
we can choose $k>0$ such that $d(x_k,
\mathcal{O}_f(y))>\epsilon_1.$ This is a contradiction. Thus if
$f$ has the weak shadowing property on $\Lambda,$ then the
shadowing point is in $\Delta.$

\end{proof}
\begin{lem}\label{weak}Let $f\in{\rm Diff}(M),$ and
let $\Delta\subset\Lambda$ be a normally hyperbolic $f$-invariant
submanifold of $M.$ Suppose that $f$ has the weak shadowing
property on $\Lambda$. Then if $\Delta$ is an arc or a disk, then
$f|_{\Delta}$ is not the identity map.
\end{lem}
\begin{proof}
Let $\Delta\subset\Lambda$ be a normally hyperbolic for $f,$ and
let $\Delta$ be an arc. Suppose that $f$ has the weak shadowing
property on $\Lambda$. We will use the method of proof by
contradiction. Assume that $f|_{\Delta}$ is an identity map. Let
$l={\rm diam}(\Delta).$ Take $\epsilon=l/4.$ Let
$0<\delta<\epsilon$ be the number of the weak shadowing property
of $f$. Then we construct $\delta$-pseudo orbit $\xi$ of $f$ as
follows; For fix $k>0$, choose distinct points $x_1, x_2, \ldots,
x_k \in\Delta$ such that
\begin{itemize}
\item[(a)] $d(x_i, x_{i+1})<\delta$ for $i=1,\ldots k-1,$
\item[(b)] $x_1=x$ and $d(x, x_k)>2\epsilon.$
\end{itemize}

Define $\xi=\{y_i\}_{i\in\Z}$ by $y_{ki+j}=x_j$ for $i\in\Z$ and
$j=0, 1, \ldots, k-1.$ Since $f$ has the weak shadowing property
on $\Lambda,$ by Lemma \ref{weaknormal}, we can find a point
$z\in\Delta$ such that $\{x_i\}_{i\in\Z}\subset
B_{\epsilon}(\mathcal{O}_f(z)).$ Since $f|_{\Delta}$ is the
indentity map, we can find $l>0$ such that $d(y_l,
\mathcal{O}_f(z))>\epsilon.$ This is a contradiction.
\end{proof}

 Let $M$ be a compact $C^{\infty}$
$n$-dimensional Riemannian manifold endowed with a volume form
$\omega$, and let $f\in{\rm Diff}_{\mu}(M).$
 To prove the results, we will use the following is the well-known
 Franks' lemma for the conservative case, stated and proved in
 \cite[Proposition 7.4]{BDP}.
 \begin{lem}\label{frank} Let $f\in {\rm Diff}^1_{\mu}(M)$, and
 $\mathcal{U}$ be a $C^1$-neighborhood of $f$ in ${\rm
 Diff}^1_{\mu}(M).$ Then there exist a $C^1$-neighborhood
 $\mathcal{U}_0\subset \mathcal{U}$ of $f$ and $\epsilon>0$ such
 that if $g\in\mathcal{U}_0$, any finite $f$-invariant set $E=\{x_1, \ldots,
 x_m\},$ any neighborhood $U$ of $E$ and any volume-preserving
 linear maps $L_j:T_{x_j}M\to T_{g(x_j)}M$ with
 $\|L_j-D_{x_j}g\|\leq\epsilon$ for all $j=1,\ldots, m,$ there is a
 conservative diffeomorphism $g_1\in\mathcal{U}$ coinciding with
 $f$ on $E$ and out of $U,$ and $D_{x_j}g_1=L_j$ for all $j=1,
 \ldots, m.$
\end{lem}

\begin{rk}\label{moser} Let $f\in{\rm Diff}^1_{\mu}(M).$ From the Moser's
Theorem(see \cite{Mo}), there is a smooth conservative change of
coordinates $\varphi_x:U(x)\to T_xM$ such that $\varphi_x(x)=0,$
where $U(x)$ is a small neighborhood of $x\in M.$
\end{rk}
\begin{pro}\label{ws} Let $M$ be a closed $C^{\infty}$ two-dimensional manifold. If $f\in int\mathcal{WS}_{\mu}(M),$ then
every periodic point of $f$ is hyperbolic.
\end{pro}
\begin{proof}Take $f\in int\mathcal{WS}_{\mu}(M),$ and
$\mathcal{U}(f)$ a $C^1$-neighborhood of $f\in
int\mathcal{E}_{\mu}(M).$ Let $\epsilon>0$ and
$\mathcal{V}(f)\subset\mathcal{U}_0(f)$ corresponding number and
$C^1$-neighborhood given by Lemma \ref{frank}. We will derive a
contradiction, we may assume that there exists a nonhyperbolic
periodic point $p\in P(g)$ for some $g\in\mathcal{V}(f).$ To
simplify the notation in the proof, we may assume that $g(p)=p.$
Then there is at least one eigenvalue $\lambda$ of $D_pg$ such
that $|\lambda|=1.$

 By making use of the Lemma \ref{frank}, we
linearize $g$ at $p$ with respect to Moser's Theorem; that is, by
choosing $\alpha>0$ sufficiently small we construct $g_1$
$C^1$-nearby $g$ such that
$$g_1(x)=\left\{%
\begin{array}{ll}
    \varphi^{-1}_p\circ D_pg\circ \varphi_p(x) &\mbox{if} \quad x\in B_{\alpha}(p), \\
    g(x) & \mbox{if}\quad x\notin B_{4\alpha}(p). \\
\end{array}%
\right.$$ Then $g_1(p)=g(p)=p.$

First, we may assume that $\lambda\in\R$ with $\lambda=1.$ Let $v$
be the associated non-zero eigenvector such that $\|v\|=\alpha/4.$
Then we can get a small arc $\mathcal{I}_v=\{tv:-1\leq t\leq
1\}\subset \varphi_p(B_{\alpha}(p)).$ Take $\epsilon_1=\alpha/8$.
Let $0<\delta<\epsilon$ be the number of the weak shadowing
property of $g_1.$ Then by our construction of $g_1,$
$\varphi^{-1}_p(\mathcal{I}_v)\subset B_{\alpha}(p).$ Then, it is
clear that $\varphi^{-1}_p(\mathcal{I}_v)$ is a normally
hyperbolic for $g_1.$ Put
$\mathcal{J}_p=\varphi_p(\mathcal{I}_v).$ For the above
$\delta>0$, we construct $\delta$-pseudo orbit
$\xi=\{x_i\}_{i\in\Z}\subset \mathcal{J}_p$ as follows; For fix
$k\in\Z,$ choose distinct points $x_0=p, x_1, x_2, \ldots, x_k$ in
$\mathcal{J}_p$ such that
\begin{itemize}
\item[(a)]$d(x_i, x_{i+1})<\delta$ for $i=0, 1, \ldots, k-1,$
\item[(b)]$d(x_{-i-1}, x_{-i})<\delta$ for $i=0, \ldots, k-1,$
\item[(c)]$x_0=x$ and $d(x_{-k}, x_k)>2\epsilon_1.$
\end{itemize}
Now we define $\xi=\{x_i\}_{i\in\Z}$ by $x_{ki+j}=x_j$ for
$i\in\Z$ and $j=-k-1, -k-2,\ldots, -1, 0, 1, \ldots, k-1.$ Since
$g_1$ has the weak shadowing property, $g_1|_{\mathcal{J}_p}$ must
have the weak shadowing property. Thus we can find a point $y\in
M$ such that $\{x_i\}_{i\in\Z}\subset
B_{\epsilon_1}(\mathcal{O}_{g_1}(y)).$ For any
$v\in\mathcal{I}_v,$ $\varphi^{-1}_p(v)\in \mathcal{J}_p\subset
B_{\alpha}(p)$ and
$$g_1(\varphi^{-1}_p(v))=\varphi^{-1}_p\circ
D_pg\circ\varphi_p(\varphi^{-1}_p(v)).$$ Then
$g_1(\varphi^{-1}_p(v))=\varphi^{-1}_p(v).$ Thus
$g^l_1({\mathcal{J}_p})=\mathcal{J}_p$ for some $l>0.$

Since $g_1$ has the weak shadowing property, by Lemma
\ref{weaknormal}, the point $y\in\mathcal{J}_p.$ But, by Lemma
\ref{weak}, the identity map does not have the weak shadowing
property. Thus $g_1|_{\mathcal{J}_p}$ does not have the weak
shadowing property.

Finally,  if $\lambda\in\C,$ then to avoid the notational
complexity, we may assume that $g(p)=p.$ As in the first case, by
Lemma \ref{frank}, there are $\alpha>0$ and $g_1\in
\mathcal{V}(f)$ such that $g_1(p)=g(p)=p$ and

$$g_1(x)=\left\{%
\begin{array}{ll}
    \varphi^{-1}_p\circ D_pg\circ \varphi_p(x) &\mbox{if} \quad x\in B_{\alpha}(p), \\
    g(x) & \mbox{if}\quad x\notin B_{4\alpha}(p). \\
\end{array}%
\right.$$

With a $C^1$-small modification of the map $D_pg$, we may suppose
that there is $l>0$(the minimum number) such that $D_pg^l(v)=v$
for any $v\in\varphi_p(B_{\alpha}(p))\subset T_pM.$ Then, we can
go on with the previous argument in order to reach the same
contradiction. Thus, every periodic point of $f\in
int\mathcal{WS}_{\mu}(M)$ is hyperbolic.
\end{proof}
For any $\epsilon>0$, let $B_{\epsilon}(\Delta)$ be the
$\epsilon$-neighborhood of $\Delta.$
\begin{lem}\label{limitweaknormal}Let $f\in{\rm Diff}(M),$ and
let $\Delta\subset\Lambda$ be a normally hyperbolic $f$-invariant
submanifold of $M.$ Suppose that $f$ has the limit weak shadowing
property on $\Lambda$. Then the shadowing point is in $\Delta.$
\end{lem}
\begin{proof}
Suppose that $f$ has the limit weak shadowing property on
$\Lambda$. Since $\Delta$ is a normally hyperbolic,  we can choose
$k>0$ and $\epsilon_1>0$ such that for any $x\in
B_{\epsilon_1}(\Delta)\setminus\Delta,$ $d(f^k(x), \Delta)>
\epsilon_1.$ Let $0<\delta<\epsilon_1$ be the number of the limit
weak shadowing property of $f|_{\Lambda}$ for $\epsilon_1.$ Since
$f$ has the limit weak shadowing property on $\Lambda,$ $f$ must
have the limit weak shadowing property on $\Delta.$ Thus for any
$\delta$-limit pseudo orbit $\{x_i\}_{\in\Z}\subset\Delta$ of $f,$
we can find a point $y\in B_{\epsilon_1}(\Delta)$ such that
$\{x_i\}_{i\in\Z}\subset B_{\epsilon_1}(\mathcal{O}_f(y))$ then
and $d(x_i, \mathcal{O}_f(y))\to\infty$ as $i\to\pm\infty.$ But,
if $y\in B_{\epsilon_1}(\Delta)\setminus \Delta$ then from the
above facts, we can choose $k>0$ such that $d(x_k,
\mathcal{O}_f(y))>\epsilon_1.$ This is a contradiction. Thus if
$f$ has the limit weak shadowing property on $\Lambda,$ then the
shadowing point is in $\Delta.$
\end{proof}
\begin{lem}\label{limitweak}Let $f\in{\rm Diff}(M),$ and
let $\Delta\subset\Lambda$ be a normally hyperbolic $f$-invariant
submanifold of $M.$ Suppose that $f$ has the limit weak shadowing
property on $\Lambda$. Then if $\Delta$ is an arc or a disk, then
$f|_{\Delta}:\Delta\to\Delta$ is not the identity map.
\end{lem}
\begin{proof}Let $\Delta\subset\Lambda$ be a normally hyperbolic
for $f,$ and let $\Delta$ be an arc. Suppose that $f$ has the
limit weak shadowing property on $\Lambda$. We will use the method
of proof by contradiction. Assume that $f|_{\Delta}$ is an
identity map. Let $l={\rm diam}(\Delta).$ Take $\epsilon=l/4.$ Let
$0<\delta<\epsilon$ be the number of the limit weak shadowing
property of $f$. Then we construct $\delta$-limit pseudo orbit
$\xi$ of $f$ as follows; For fix $k>0$, choose distinct points
$x_0, x_1, x_2, \ldots, x_k \in\Delta$ such that
\begin{itemize}
\item[(a)] $d(x_i, x_{i+1})<\delta$ for $i=0,\ldots, k-1,$
\item[(b)] $d(x_{-i-1}, x_{-i})<\delta$ for $i=0, \ldots, k-1,$
\item[(c)] $x_{-k-j}=x_{-k}$ for $j\geq0,$ and $x_{k+j}=x_k$
for $j\geq0.$
\end{itemize}

 Then $\xi=\{\ldots, x_{-k}, x_{-k}, x_{-k+1}, \ldots, x_{-1}, x,
 x_1, \ldots, x_k, x_k, \ldots, \}$ is a $\delta$-limit pseudo
 orbit of $f.$ Clearly, $\xi\subset\Delta.$  Since $f$ has the weak shadowing property
on $\Lambda,$ by Lemma \ref{limitweaknormal}, we can find a point
$z\in\Delta$ such that $\{x_i\}_{i\in\Z}\subset
B_{\epsilon}(\mathcal{O}_f(z)),$ and $d(x_i,
\mathcal{O}_f(z))\to\infty,$ as $i\to\pm\infty.$ Since
$f|_{\Delta}$ is the indentity map, we can find $l>0$ such that
$d(y_l, \mathcal{O}_f(z))>\epsilon.$ This is a contradiction.

\end{proof}

\begin{pro}\label{lws}Let $M$ be a closed $C^{\infty}$ two-dimensional manifold. If $f\in int\mathcal{LWS}_{\mu}(M)$ then
every periodic point of $f$ is hyperbolic.
\end{pro}
\begin{proof}Take $f\in int\mathcal{LWS}_{\mu}(M),$ and
$\mathcal{U}(f)$ a $C^1$-neighborhood of $f\in
int\mathcal{LWS}_{\mu}(M).$ Let $\epsilon>0$ and
$\mathcal{V}(f)\subset\mathcal{U}_0(f)$ corresponding number and
$C^1$-neighborhood given by Lemma \ref{frank}. To derive a
contradiction, we may assume that there exists a nonhyperbolic
periodic point $p\in P(g)$ for some $g\in\mathcal{V}(f).$ To
simplify the notation in the proof, we may assume that $g(p)=p.$
Then as in the proof of Proposition \ref{ws}, we can take
$\alpha>0$ sufficiently small, and a smooth map $\varphi_p:
B_{\alpha}(p)\to T_pM$. Form the above construction, we can make
an arc $\mathcal{J}_p\subset B_{\alpha}(p)$ and for
$g_1\in\mathcal{V}(f)$, $\mathcal{J}_p$ is a $g_1$-invariant
normally hyperbolic.  Take $\epsilon_1=({\rm
length}{\mathcal{J}_p})/4,$ let $0<\delta<\epsilon_1$ be the
number of the limit weak shadowing property for $g_1.$ Form now,
we construct $\delta$-limit pseudo orbit of $g_1$ as follows; For
fix $k>0$, choose distinct points $x_0=0, x_1,\ldots,
x_k\in\mathcal{J}_p$ such that
\begin{itemize}
\item[(a)] $d(x_i, x_{i+1})<\delta$ for $i=0, \ldots, k-1,$
\item[(b)] $d(x_{-k-i}, x_{-i})<\delta$ for $i=0, \ldots, k-1,$
\item[(c)] $x_0=x$ and $d(x_{-k}, x_k)>2\epsilon_1,$
\item[(d)] $x_{-k-j}=x_{-k}$ and $x_{k+j}=x_k$ for $j\geq0.$
\end{itemize}
Then $\xi=\{\ldots, x_{-k}, x_{-k}, x_{-k+1}, \ldots, x_{-1},
x_0=p, x_1, \ldots, x_k, x_k\ldots,\}$ is a $\delta$-limit pseudo
orbit of $g_1$ and $\xi\subset\mathcal{J}_p.$ Since
$\mathcal{J}_p$ is a normally hyperbolic for $g_1,$ by Lemma
\ref{limitweaknormal}, the shadowing point $y\in\mathcal{J}_p.$
Since $g_1|_{\mathcal{J}_p}$ is the identity map, By Lemma
\ref{limitweak}, $g_1$ does not have the limit weak shadowing
property on $\mathcal{J}_p.$ This is a contradiction.

Finally, if $\lambda\in\C,$ then as in the proof of Proposition
\ref{weak}, for $g_1\in\mathcal{V}(f),$ we can take $l>0$ such
that $D_pg^l_1(v)=v$ for any $v\in\varphi_p(B_{\alpha}(p))\subset
T_pM.$ Then from the previous argument in order to reach the same
contradiction. Thus, every periodic point of $f\in
int\mathcal{LWS}_{\mu}(M)$ is hyperbolic.

\end{proof}


\end{document}